\documentclass[12pt]{amsart}
\usepackage{epsfig}
\usepackage{amssymb, amsmath}
\usepackage{amsfonts,graphics,epsfig}
\usepackage{mathtools}
\usepackage{mathrsfs}
\usepackage{pb-diagram, pb-xy}
\usepackage[all]{xy}
\usepackage{comment}
\usepackage{color}
\usepackage{tikz-cd}
\usepackage{setspace}
\usepackage[left=1.5in,right=1.5in,top=1.5in,bottom=1.5in]{geometry}
\usepackage[colorlinks=true, linkcolor=red, urlcolor=blue, citecolor=blue]{hyperref}

\newcommand{\mbR}{\mathbb{R}}

\newcommand{\mbZ}{\mathbb{Z}}
\newcommand{\mbQ}{\mathbb{Q}}

\def\mbN{\mathbb{N}}
\def\mbP{\mathbb{P}}

\def\mfD{\mathfrak{D}}
\newcommand{\<}{\leq}
\def\>{\geq}

\def\ve{\varepsilon}

\def\subset{\subseteq}

\def\implies{\Rightarrow}

\newcommand{\lrd}{\lfloor}
\newcommand{\rrd}{\rfloor}

\newcommand{\num}{\equiv}

\newcommand{\bir}{\dashrightarrow}
\def\mcO{\mathcal{O}}

\def\msA{\mathscr{A}}

\def\mfI{\mathfrak{T}}

\newtheorem{theorem}{Theorem}[section]
\newtheorem{lemma}[theorem]{Lemma}
\newtheorem{proposition}[theorem]{Proposition}
\newtheorem{corollary}[theorem]{Corollary}
\newtheorem{conjecture}[theorem]{Conjecture}
\newtheorem{claim}[theorem]{Claim}

\theoremstyle{remark}
\newtheorem{remark}[theorem]{Remark}
\theoremstyle{definition}
\newtheorem{definition}[theorem]{Definition}
\theoremstyle{definition}

\numberwithin{equation}{section}

\def\deg{\operatorname{deg}}

\def\Supp{\operatorname{Supp}}
\def\dim{\operatorname{dim}}

\def\chr{\operatorname{char}}

\def\lct{\operatorname{lct}}
\def\LCT{\operatorname{LCT}}

\def\min{\operatorname{min}}

\def\mult{\operatorname{mult}}

\def\vol{\operatorname{vol}}
\def\Span{\operatorname{Span}}

\author{Omprokash Das}
\address{Department of Mathematics\\
University of California, Los Angeles\\
520 Portola Plaza\\
Math Sciences Building 6363.}
\email{omprokash@gmail.com, das@math.ucla.edu}

\date{}

\begin{document}
\title[On the Weak--BAB Conjecture for $3$-folds in char $p>5$]{On the boundedness of anti-canonical volumes of singular Fano $3$-folds in characteristic $p>5$}
\begin{abstract}
	In this article we prove the following version of the Weak-BAB conjecture for $3$-folds in $\chr p>5$: Fix a DCC set $I\subset [0, 1)$ and an algebraically closed field $k$ of characteristic $p>5$. Let $\mfD$ be a collection of klt pairs $(X, \Delta)$ satisfying the following properties: (1) $X$ is a projective $3$-fold, (2) $\Delta$ is an $\mbR$-divisor with coefficients in $I$, (3) $K_X+\Delta\num 0$, and (4) $-K_X$ is ample. Then the set $\{\vol_X(-K_X) \ | \ (X, \Delta)\in\mfD\mbox{ for some }\Delta\}$ is bounded from above.
\end{abstract}

\maketitle

\tableofcontents

\section{Introduction}
Given a smooth projective variety $X$, the minimal model program predicts that $X$ is birational to a variety $Y$ with canonical singularities such that either $K_Y$ is ample, or $Y$ admits a fibration whose general fibers are Calabi-Yau varieties or Fano varieties. In other words, one may say that, birationally, every variety is in some sense constructed from varieties $X$ with good singularities such that either $K_X$ is ample or $K_X\num 0$ or $-K_X$ is ample. So it is quite natural to study such special varieties with the hope of obtaining some sort of classification theory. One such classification is finding the \emph{moduli space} of a given class of varieties. In this article we are interested in the moduli problem of Fano varieties i.e., $-K_X$ is ample, of dimension $3$ in positive characteristic. The first problem in this direction is proving the boundedness of the \emph{moduli functor}, i.e., finding a fibration $f:\mathcal{X}\to T$ such that every Fano $3$-fold appears as closed fiber of $f$. However, this turns out to be a too general of a question to be true without some restriction on the singularities of $X$; counterexamples are known to exists even in characteristic $0$. This leads to the following conjecture of Borisov, Alexeev and Borisov, known as the BAB conjecture.
\begin{conjecture}[BAB Conjecture]\label{con:bab-conjecture}
	Fix a positive integer $n$ and a real number $\ve>0$. The the set of all projective varieties $X$ satisfying the following properties: 
	\begin{enumerate}
		\item the dimension of $X$ is $n$,
		\item $(X, \Delta)$ has $\ve$-log canonical singularities for some boundary $\mbR$-divisor $\Delta$,
		\item $-(K_X+\Delta)$ is ample,
	\end{enumerate}
	form a bounded family.
\end{conjecture}
Note that a pair $(X, \Delta)$ is said to have $\ve$-log canonical singularities if the discrepancies satisfy $a(E; X, \Delta)\>-1+\ve$ for all divisors $E$ over $X$. A necessary condition that follows from the BAB conjecture is that the volumes of $-K_X$ (see Definition \ref{def:volumes}) is bounded from above; this is known as the Weak-BAB conjecture.
\begin{conjecture}[Weak-BAB Conjecture]\label{con:weak-bab-conjecture}
	Fix a positive integer $n$ and a real number $\ve>0$. Let $\mfD$ be the set of all log pairs $(X, \Delta)$ satisfying the following properties: 
	\begin{enumerate}
		\item $X$ is a projective variety of dimension $n$,
		\item $(X, \Delta)$ has $\ve$-log canonical singularities for some boundary $\mbR$-divisor $\Delta$, and
		\item $-(K_X+\Delta)$ is ample.
	\end{enumerate}
	Then there exists a positive real number $M(n, \ve)$ depending only on $n$ and $\ve>0$ such that the set $\{\vol_X(-K_X)\; | \; (X, \Delta)\in\mfD \mbox{ for some }\Delta \}$ is bounded from above by $M(n ,\ve)$.
\end{conjecture}
We note that the BAB conjecture (and hence the Weak-BAB conjecture) is known in dimension $2$ in \emph{arbitrary characteristic} due to Alexeev, \cite{Ale94}.
In dimension $3$ the following partial results on the BAB conjecture were known for a while:  The toric Fano $3$-fold case was proven by Borisov brothers in \cite{BB92}. In \cite{KMM92}, the authors proved the conejcture for smooth Fano varieties in char $0$; Kawamata in \cite{Kaw92} proved the conjecture in char $0$ for $\mbQ$-factorial terminal Fano $3$-folds of Picard number $\rho(X)=1$. In \cite{KMMT00}, the authors proved the same conjecture for $\mbQ$-factorial Fano $3$-folds with canonical singularities in char $0$.\\ 
The Weak-BAB conjecture is known in dimension $3$ in char $0$ when the Picard number $\rho(X)=1$, due to \cite{Lai16}, and the general case due to Jiang, \cite{Jia14nov}. More recently the BAB conjecture in full generality (and hence the Weak-BAB conjecture) is completely proved in char $0$ in every dimension in a series of breakthrough papers by Birkar, \cite{Bir16mar, Bir16sep}. On the other hand, very little is known about either of these two conjectures in positive characteristic in dimension $3$ or higher. In this article we prove a special case of the Weak-BAB conjecture for $3$-folds in characteristic $p>5$. More specifically, using the ideas from \cite{HMX14} we prove the following results:
\begin{theorem}\label{thm:main-theorem}
	Fix an algebraically closed field $k=\overline{k}$ of characteristic $p>5$ and a DCC set $I\subset [0, 1)$. Let $\mfD$ be the set of all klt pairs $(X, \Delta)$, where
	\begin{enumerate}
	\item $X$ is a projective $3$-fold,
	\item the coefficients of $\Delta$ belong to $I$,
	\item $K_X+\Delta\num 0$, and 
	\item $-K_X$ is big, 
\end{enumerate}
then the set
	\[
		\{\vol_X(-K_X)\ |\ (X, \Delta)\in\mfD \mbox{ for some }\Delta\}
	\]
	is bounded from above.\\
\end{theorem}
\begin{remark}
The theorem above is a positive characteristic ($p>5$) analog of Theorem B in \cite{HMX14} in dimension $3$. It is also worthwhile to note that the statement above is slightly more general ($-K_X$ is big) than the usual BAB conjecture (which assumes that $-K_X$ is ample).\\  
\end{remark}

We note that, as far as we know, except \cite{Zhu17nov}, our result is the only result on the Weak-BAB conjecture for $3$-folds in positive characteristic. About the paper \cite{Zhu17nov}, in this article the author proves a boundedness result for the anti-canonical volumes of Fano $3$-folds in $\chr p>5$ under the restrictive assumption that the \emph{Seshadri constant} of $-K_X$ is larger than some fixed number. We note that this kind of restriction is \emph{not standard} in the context of moduli problems. Seshadri constants typically do not play an important role in the theory of moduli spaces and do not seem to appear naturally in the context of the minimal model program. On the other hand, our hypothesis involving DCC sets is a standard hypothesis which appears quite naturally and frequently in various statements related to the minimal model program and the moduli problem in general, for example, see \cite[Corollary 1.7]{HMX14} for a boundedness result of Fano varieties (the BAB conjecture) in characteristic $0$ involving DCC sets. In this sense our result is an important first step towards the proof of the boundedness of Fano $3$-folds in characteristic $p>5$.\\ 
 
\textbf{Idea of the Proof of Theorem} \ref{thm:main-theorem}: The intuitive idea of the proof is the following: Since the coefficients of $\Delta$ are contained in a fixed DCC set $I$, we can find an $\ve>0$ depending only on the set $I$ and satisfying the following properties: if $(X, \Delta)\in\mfD$ and $\Phi\>0$ is an effective $\mbR$-Cartier divisor on $X$ such that $K_X+\Phi\num 0$ and $\Phi$ is contained in the ``$\ve$-neighborhood'' of $\Delta$, then $(X, \Phi)$ has klt singularities. Now if the $\vol_X(-K_X)$ is not bounded above as $X$ varies in $\mfD$, then choose a pair $(X, \Delta)\in\mfD$ such that $\vol_X(\Delta)=\vol_X(-K_X)>n^n$, where $n=\dim X$. Then we can construct an effective $\mbR$-divisor $\Psi\sim_\mbR\Delta$ contained in the ``$\ve$-neighborhood'' of $\Delta$ such that $(X, \Psi)$ is not klt, which is a contradiction.\\

{\bf Acknowledgement.} I would like to thank Joe Waldron for answering several of my questions. I would also like to thank Burt Totaro for his valuable comments.
 My sincerest gratitude goes to Christopher Hacon for answering many questions and fruitful discussions. I would also like to thank Akash Sengupta for pointing out an error in the previous version.\\    

\section{Preliminaries}
Throughout the article by an \textbf{\emph{arbitrary field $k$}}, we mean that the characteristic of $k$ is either $0$ or positive and $k$ is possibly \emph{imperfect}; in particular $k$ is not necessarily \emph{algebraically closed}. 
 
\begin{definition}\label{def:varieties}
	Let $k$ be an arbitrary field. A \emph{variety} $X$ over $k$ is an integral separated scheme of finite type over $k$. A \emph{curve} over $k$ is a variety of dimension $1$ over $k$. A \emph{surface} over $k$ is a variety of dimension $2$ over $k$
\end{definition}

\begin{definition}\label{def:mmp-singularities}
	Let $k$ be an arbitrary field and $X$ a \emph{normal} variety over $k$. Let $\Delta$ be an $\mbR$-divisor on $X$. We say that $\Delta$ is a \emph{boundary divisor} if the coefficients of $\Delta$ belong to the closed interval $[0, 1]$. If $\Delta$ is a boundary divisor, then a pair $(X, \Delta)$ is called a \emph{log pair} if  $K_X+\Delta$ is \emph{$\mbR$-Cartier}. For a log pair $(X, \Delta)$, we define \emph{terminal, canonical, klt, plt, dlt and log canonical or lc} singularities as in \cite[Definition 2.8]{Kol13}. We emphasize that in this article whenever we talk about singularities of a pair $(X, \Delta)$ we always assume that $\Delta$ is an \emph{effective} $\mbR$-divisor.  
\end{definition}

\begin{remark}\label{rmk:dim-1-singularities}
	Let $C$ be a normal curve over an arbitrary field $k$ and $\Delta$ is a boundary divisor on $C$. Then $(C, \Delta)$ is log canonical (resp. klt) if and only if the coefficients of $\Delta$ are less than or equal to $1$ (resp. strictly less than $1$).\\
\end{remark}

\begin{remark}
	We note that the Nakai-Moishezon-Klemian criterion for ampleness, Kodaira's lemma for big divisors, etc. hold on projective varieties defined over arbitrary fields. For more details see \cite[Remark 2.3 and 2.4]{Tan18}.  
\end{remark}

\begin{remark}
	For two $\mbR$-divisors $D$ and $D'$, by $D>D'$ we mean that $D\>D'$ and $D\neq D'$.
\end{remark}

\begin{definition}\label{def:degree-on-curve}
Let $C$ be a normal, i.e., regular curve over an arbitrary field $k$. Let $D=\sum r_iP_i$ be an Weil $\mbR$-divisor on $C$. Then the degree of $D$ on $C$ over $k$ is defined as
\[
	\deg_k(D):=\sum_i [k(P_i):k]\cdot r_i,
\]
where $k(P_i)$ is the residue field of the \emph{closed} point $P_i\in C$ and $[k(P_i):k]$ is the extension degree of the fields.\	
\end{definition}

\subsection{DCC sets and adjunction}
\begin{definition}\label{def:dcc-sets}
	We say that a subset $I$ of real numbers satisfies the \emph{descending chain condition} or DCC if for every decreasing sequence $\{a_i\}\subset I$, i.e., $a_i\>a_{i+1}$ for all $i\>1$, there exists a $N\gg 0$ such that $a_i=a_{i+1}$ for all $i\>N$; equivalently, $I$ does not contain any infinite strictly decreasing sequence.\\
	
	Let $I\subset [0, 1]$. We define
	\[
		I_+:=\{0\}\cup\left\{j\in[0,1]\; \middle|\; j=\sum_{p=1}^\ell i_p, \mbox{ for some } i_1, i_2,\ldots, i_\ell\in I \right\}
	\] 
	and
	\[
		D(I):=\left\{a\<1\;\middle|\; a=\frac{m-1+f}{m}, m\in\mbN, f\in I_+\right\}.
	\]
\end{definition}

The following lemma shows some useful properties of DCC sets.
\begin{lemma}\label{lem:dcc-sets}
Let $I, I_1, I_2,\ldots, I_n$ be subsets of $\mbR_{\>0}$.
\begin{enumerate}
	\item Any subset of a DCC set is a DCC set.
	\item If $I_1, I_2,\ldots, I_n$ all satisfy the DCC, then $\cup_{i=1}^n I_n$ satisfies the DCC.
	\item If $I$ satisfies the DCC and $r_1, r_2,\ldots, r_k\>0$, then $I'=\{ar_i:a\in I, 1\<i\<k\}$ satisfies the DCC.
	\item If $I_1, I_2,\ldots, I_n$ are all DCC sets, then $\sum_{i=1}^n I_i:=\{\sum_{i=1}^n r_i: r_i\in I_i\}$ is a DCC set.
	\item Fix a positive integer $N>0$. If $I$ is a DCC set, then $\Span_N(I):=\left\{\sum_{i=1}^Nn_ir_i: n_i\in\mbZ_{\>0}, r_i\in I\mbox{ for all } i\right\}$ is a DCC set.
	\item Let $I\subset [0, 1]$, then $I$ satisfies the DCC if and only if $D(I)$ satisfies the DCC.
\end{enumerate}	
\end{lemma}

\begin{proof}
	Part (1) is obvious. For part (2), by contradiction assume that there is a strictly decreasing sequence $\{t_m\}$ contained in $\cup_{i=1}^nI_i$. Then by passing to an infinite subsequence we may assume that all the terms of $\{t_m\}$ are contained in $I_j$ for some fixed $j$ satisfying $1\<j\<n$. This is a contradiction to the DCC property of $I_j$. Part (3) follows from part (1) and (2) by noticing that $I'=\cup_{i=1}^kJ_i$, where $J_i=\{ar_i: a\in I\}$, and that $J_i$ satisfies DCC for all $i$.\\

For part (4), by contradiction assume that there is a strictly decreasing sequence $\{r_k\}$ contained in $\sum_{i=1}^nI_n$. Let $r_k=\sum_{j=1}^n r_{kj}$, where $r_{kj}\in I_j$ for all $k$ and $j$. Since $\{r_{kj}\}_{k\>1}$ satisfies DCC for all $j$, by passing to subsequences with common indices we may assume that each $\{r_{kj}\}_{k\>1}$ is a monotonically increasing sequence for all $j$. It then follows that $r_k\<r_{k+1}$ for all $k\>1$, which is a contradiction, since $r_k>r_{k+1}$ for all $k\>1$.\\

For part (5) let's define $\mbZ_{\>0}\cdot I\:=\{nr: n\in\mbZ_{\>0}, r\in I\}$. Then 
\[\Span_N(I)=\mbZ_{\>0}\cdot I+\mbZ_{\>0}\cdot I+\cdots_{ (N\mbox{ times})}+\mbZ_{\>0}\cdot I.\]
 Thus (5) will follow from (4) if we can show that $\mbZ_{\>0}\cdot I$ is a DCC set. To that end, by contradiction assume that there is a strictly decreasing sequence $\{n_ir_i\}$ in $\mbZ_{\>0}\cdot I$, i.e., 
\begin{equation}\label{eqn:span-dcc-contradiction}
	n_1r_1>n_2r_2>n_3r_3>\cdots.
\end{equation}	
Since $\{r_i\}$ is contained in a DCC set $I$, by passing to an infinite subsequence we may assume that  $\{r_i\}$ is a monotonically increasing sequence, i.e.,
\begin{equation}\label{eqn:given-dcc}
	r_1\<r_2\<r_3\<\cdots
\end{equation}
Therefore we have $r_i\>r_1>0$ for all $i\>1$. From \eqref{eqn:span-dcc-contradiction} we also have that $n_1r_1\>n_ir_i$ for all $i\>1$. Thus we get that $n_1\>n_i$ for all $i\>1$. In particular, $\{n_i\}$ is a bounded sequence of positive integers, hence a finite set. Thus by passing to an infinite subsequence of \eqref{eqn:span-dcc-contradiction} we may assume that $n_i=n_{i+1}$ for all $i\>1$, which gives a contradiction to \eqref{eqn:given-dcc}.\\  

For part (6) see \cite[Lemma 4.4]{MP04}.\\

\end{proof}

\subsection{Adjunction} In this subsection we collect some results about adjunction to codimension $1$ subvarieties.
\begin{proposition}[Different]\label{pro:adjunction}
	Fix a DCC set $I\subset [0, 1]$. Let $(X, \Delta)$ be a log pair defined over an arbitrary field $k$ and $S$ a component of $\lrd\Delta\rrd$. Let $S^n\to S$ be the normalization morphism. Then there exists a canonically determined effective $\mbR$-divisor $\Delta_{S^n}\>0$ on $S^n$ such that
	\[
		(K_X+\Delta)|_{S^n}\sim_\mbR K_{S^n}+\Delta_{S^n}.
	\]

Moreover, if $(X, \Delta)$ is log canonical outside a codimension $3$ closed subset and the coefficients of $\Delta$ belong to $I$, then the coefficients of $\Delta_{S^n}$ belong to $D(I)$. More precisely: write $\Delta=S+\sum_{i\>2}d_iD_i$, let $P'$ be a prime Weil divisor on $S^n$ and $P$ its image on $S$. Then there exists $m\in\mbN\cup\{\infty\}$ depending only on $X, S$ and $P$, there are non-negative integers $l_i\in\mbZ_{\>0}$ depending only on $X, S, D_i$ and $P$ such that the coefficient of $P'$ in $\Delta_{S^n}$ is 
\[
	\frac{m-1}{m}+\sum_{i\>2}\frac{l_id_i}{m}.
\]
	\end{proposition}

\begin{proof}
	The proof is same as the proofs of Proposition 4.1 and 4.2 in \cite{Bir16}. We note that the proof in \cite{Bir16} essentially reduces the problem to a computation on excellent surfaces. Since any scheme of finite type over an arbitrary field is an excellent scheme and the local rings of excellent schemes are again excellent, the same proof works in our case. 
\end{proof}

\begin{lemma}[Easy Adjunction]\label{lem:easy-adjunction}
	Let $X$ be a normal variety over an arbitrary field $k$. Let $S$ be a prime Weil divisor and $\Delta\>0$ an effective $\mbR$-divisor on $X$ such that $S$ is not contained in the support of $\Delta$ and $K_X+S+\Delta$ is $\mbR$-Cartier. Let $S^n\to S$ be the normalization morphism and $(S^n, \Delta_{S^n})$ is defined by adjunction $K_{S^n}+\Delta_{S^n}\sim_\mbR (K_X+S+\Delta)|_{S^n}$.\\
	If $(X, S+\Delta)$ is terminal, canonical, plt or lc, then so is $(S^n, \Delta_{S^n})$, respectively.  
\end{lemma}

\begin{proof}
	It follows from \cite[Lemma 4.8]{Kol13}.\\
\end{proof}

\section{Lemmas and Propositions} 
\label{sec:lem-and-pro}

\subsection{The volume}
\begin{definition}\label{def:volumes}
	Let $X$ be a projective variety of dimension $n$ over an algebraically closed field $k$. Let $D$ be a $\mbR$-Cartier divisor.
	Then the \emph{volume} of $D$ is defined as
	\[
		\vol_X(D):=\limsup_{m\to+\infty}\frac{n!\dim_kH^0(X, \mcO_X(\lrd mD\rrd))}{m^n}.
	\] 
\end{definition}
It is know that $D$ is big if and only if $\vol_X(D)>0$.\\

In the following lemmas we establish some perturbation techniques for log pairs $(X, \Delta)$ using a divisor $D$ such that $\vol_X(D)>n^n$.

\begin{lemma}\label{lem:creating-singular-divisors}
	Let $X$ be a proper variety of dimension $n$ defined over an algebraically closed field $k$ of arbitrary characteristic. Let $M$ be a big $\mbQ$-Cartier divisor on $X$ such that $\vol_X(M)>n^n$. Then for every smooth closed point $x\in X$, there exists an effective $\mbQ$-Cartier divisor $D=D(x)\>0$ such that $D\sim_\mbQ M$ and 
	\[
		\mult_x D>n.
	\]
\end{lemma}
\begin{proof}
	The following proof is based on the proof of \cite[Lemma 6.1]{Kol97}.\\
Since $\vol_X(M)>n^n$, there exists a small positive rational number $\delta\in\mbQ^+$ such that $\vol_X(M)>(n+\delta)^n$. Let $t>0$ be a positive integer such that $tM$ is a Cartier divisor. Let $m_x\subset\mcO_X$ be the ideal sheaf of $\{x\}\subset X$. For a positive integer $s>0$ consider the following exact sequence
	\[
		\xymatrix{0\ar[r] &  m_x^s\otimes\mcO_X(tM)\ar[r] & \mcO_X(tM)\ar[r] & (\mcO_X/m_x^s)\otimes\mcO_X(tM)\cong\mcO_X/m_x^s\ar[r] & 0.}
	\]
	
Then observe that
\begin{equation}\label{eqn:dimension-inequality}
	h^0(X, m_x^s\otimes\mcO_X(tM))>0 \quad \mbox{if}\quad h^0(X, \mcO_X(tM))>h^0(X, \mcO_X/m_x^s).
\end{equation}
Let $\{x_1, x_2\ldots, x_n\}$ be a local coordinate system around $x\in X$, i.e., it is a $k(x)$-basis of the vector space $m_x/m_x^2$. Then we have 
\begin{equation}\label{eqn:dim-calculation}
	h^0(X, \mcO_X/m^s)=\dim_kk[x_1, x_2,\ldots, x_n]/(x_1, x_2,\ldots, x_n)^s=\binom{n+s-1}{n}=\frac{s^n}{n!}+O(s^{n-1}).
\end{equation}

Since $\vol_X(M)>(n+\delta)^n$, from the definition of volume it follows that $h^0(X, \mcO_X(tM))>\frac{((n+\delta)t)^n}{n!}$ for infinitely many values of $t$ which are sufficiently large and divisible. Thus from \eqref{eqn:dim-calculation} we see
that $h^0(X, \mcO_X(tM))>h^0(X, \mcO_X/m_x^s)$ for $s, t$ sufficiently large and divisible satisfying $t(n+\delta)>s$, where $n=\dim X$. In particular, from \eqref{eqn:dimension-inequality} it follows that $h^0(X, m_x^s\otimes\mcO_X(tM))>0$ for $s, t$ sufficiently large and divisible satisfying $t(n+\delta)>s$.\\

Now let $\ve\in\mbQ^+$ be a small positive rational number satisfying $0<\ve<\delta$. Choose $t$ sufficiently large and divisible so that the open interval $( t(n+\delta-\ve), t(n+\delta))$ contains a positive integer, say $s>0$, i.e., $t(n+\delta-\ve)<s<t(n+\delta)$, i.e., $n+\delta-\ve<s/t<n+\delta$. Let $D(s, t, x)$ be the divisor of zeros of a non-zero global section of $m_x^s\otimes\mcO_X(tM)$ and $D(x)=D(s, t, x)/t$. Then $\mult_xD(x)\>s/t>n+\delta-\ve>n$. 
	
\end{proof}

\begin{lemma}\label{lem:klt-perturbation}
Let $X$ be a normal projective variety of dimension $n$ defined over an algebraically closed field $k$. Let $(X, \Delta)$ be a log pair and $D$ a $\mbR$-Cartier divisor such that $\vol_X(D)>n^n$. Then for every \emph{smooth} closed point $x\in X$, there exists an effective $\mbR$-Cartier divisor $\Pi\sim_\mbR D$ passing through $x\in X$ such that $(X, \Delta+\Pi)$ is not klt at $x\in X$.  	
\end{lemma}

\begin{proof}
	
	First assume that $D$ is a $\mbQ$-Cartier divisor. Then by Lemma \ref{lem:creating-singular-divisors}, there exists an effective $\mbQ$-Cartier divisor $0\<\Pi\sim_\mbQ D$ such that $\mult_x\Pi>n$. Then by blowing up $X$ at $x$ it is easy to see that $(X, \Delta+\Pi)$ is not klt.
	
		
		Now consider the case when $D$ is an $\mbR$-Cartier divisor. Since the volume is a continuous function (see \cite[Theorem 2.2.44]{Laz04a}), there exists an effective $\mbQ$-Cartier divisor $D'\>0$ sufficiently close to $D$ such that $D\>D'$ and $\vol_X(D')>n^n$. Then by what we have just proved there exists an effective $\mbQ$-Cartier divisor $\Pi'\sim_\mbQ D'$ such that $(X, \Delta+\Pi')$ is not klt. Let $D=D'+E'$ and $\Pi=\Pi'+E'$, where $E'$ is an effective divisor. Then $\Pi\sim_\mbR D$ and $(X, \Delta+\Pi)$ is not klt.\\

\end{proof}

The following two lemmas (\ref{lem:integral-generic-fiber} and \ref{lem:singularities-generic-fiber}) show what kind of properties of the total space of a fibration $f:X\to Y$ transfer to its \emph{generic fiber} $X_\eta$.

\begin{lemma}\cite[Lemma 2.20]{BCZ18}\label{lem:integral-generic-fiber}
	Let $f:X\to Y$ be a dominant morphism of finite type between two integral schemes of finite type over a field $k$ of arbitrary characteristic. Let $\eta$ be the generic point of $Y$, and $X_\eta$ the generic fiber. Then the following statements hold:
	\begin{enumerate}
		\item $X_\eta$ is an integral scheme.
		\item $K(X_\eta)\cong K(X)$.
		\item If $x'$ is a point in $X_\eta$ and $x$ its image in $X$ through the set theoretic inclusion $X_\eta\subset X$, then $\mcO_{X_\eta, x'}\cong\mcO_{X, x}$.
		\end{enumerate}
In particular, if $X$ is normal (resp. regular, resp. $\mbQ$-factorial), then $X_\eta$ normal (resp. regular, resp. $\mbQ$-factorial).\\
		
	\end{lemma}

\begin{lemma}\cite[Corollary 2.2]{DW18}\label{lem:singularities-generic-fiber}
	Let $f:X\to Y$ be a dominant morphism between two varieties with $X$ normal. Let $\eta$ be the generic point of $Y$ and $X_\eta$ the generic fiber. Further assume that $(X, \Delta)$ is a pair such that $K_X+\Delta$ is $\mbR$-Cartier. If $(X, \Delta)$ has terminal, canonical, klt, plt, dlt or lc singularities, then the pair $(X_\eta, \Delta|_{X_\eta})$ has terminal, canonical, klt, plt, dlt or lc singularities, respectively.\\
	\end{lemma}

\begin{lemma}\label{lem:convex-lct}
	Let $(X, \Delta)$ be a klt pair of dimension $3$ and $\Phi$ an $\mbR$-divisor such that the pair $(X, \Phi)$ is not-log canonical. Then there exists $\lambda\in (0, 1)$ such that $(X, (1-\lambda)\Delta+\lambda\Phi)$ is log canonical but not klt.
\end{lemma}

\begin{proof}
	
	Let $f:Y\to X$ be a log resolution of $(X, \Delta+\Phi)$. Consider the pair $(X, (1-t)\Delta+t\Phi)$. Let $\{E_i\}$ be the set of all exceptional divisors of $f$ and also the strict transform of the components of $\Delta$ and $\Phi$. Let $a_i(t)$, $a_i(E_i; X, \Delta)$ and $a_i(E_i; X, \Phi)$ be the discrepancy of $E_i$ with respect to the pairs $(X, (1-t)\Delta+t\Phi), (X, \Delta)$ and $(X, \Phi)$, respectively. Then it is clear that $a_i(t)=(1-t)a_i(E_1; X, \Delta)+ta_i(E_i; X, \Delta)$. Therefore $a_i(t)$ is either a constant polynomial or linear polynomial in $t$. Note that $a_i(t)$ will be a constant polynomial if and only if $a_i(E; X, \Delta)=a_i(E; X, \Phi)$. Since $(X, \Delta)$ is klt and and $(X, \Phi)$ is not log canonical, if follows that there exists at least one $i$ such that $a_i(t)$ is a polynomial of degree $1$. Moreover, from the same hypothesis it also follows that $a_i(0)>-1$ for all $i$ and $a_j(1)<-1$ for some $j$. Without loss of generality assume that $a_i(t)$ is a polynomial of degree $1$ for all $i$. Then from the graphs of $a_i(t)$'s (which are straight lines) it is clear that that there exists a $t=\lambda\in(0, 1)$ such that $a_i(\lambda)\>-1$ for all $i$ and $a_j(\lambda)=-1$ for at least one $j$. In particular, $(X, (1-\lambda)\Delta+\lambda\Phi)$ is log canonical but not klt for some $\lambda\in (0, 1)$.

\end{proof}

\begin{lemma}\label{lem:q-gorenstein}
Let $(X, \Delta)$ be a log canonical pair of dimension $2$ defined over an algebraically closed field $k$. Then $X$ is $\mbQ$-Gorenstein, i.e., $K_X$ is $\mbQ$-Cartier.	
\end{lemma}

\begin{proof}
	Since $(X, \Delta)$ is log canonical, $(X, 0)$ is numerically log canonical (see \cite[4.1]{KM98} for the definition). Then by \cite[Proposition 6.3(b)]{FT12}, $(X, 0)$ is log canonical, i.e., $K_X$ is $\mbQ$-Cartier.\\
\end{proof}

\section{Log canonical thresholds}

\begin{definition}\label{def:acc-for-lct}
Let $I$ and $J$ be two sets such that $I\subset [0, 1]$ and $J\subset\mbR_{\>0}$. Let $\mfI_n(I)$ be the set of all log pairs $(X, \Delta)$ of dimension $n$ over arbitrary fields $k$ such that $(X, \Delta)$ is log canonical and the coefficients of $\Delta$ belong to $I$. Let $M$ be an effective $\mbR$-Cartier divisor on $X$. Then we define
\[
	\lct(X, \Delta; M)=\sup\{t\in\mbR\ |\ (X, \Delta+tM) \mbox{ is log canonical} \},
\]
and 
\[
	\LCT_n(I, J)=\{\lct(X, \Delta; M)\ |\ (X, \Delta)\in\mfI_n(I) \mbox{ and the coefficients of } M \mbox{ belong to } J \}.
\]
\end{definition}

\begin{remark}\label{rmk:lct}
	In the defintion above we allow the possibility that there can be two different pairs $(X, \Delta)$ and $(X', \Delta')$ contained in $\mfI_{n}(I)$ such that $X$ is defined over a field $k$ and $X'$ over $k'$ but $k$ is \emph{not isomorphic} to $k'$. This does not create any problem since from the defintion of log canonical thresholds (lct) it clearly follows that the lct does not depend on the base field of the ambient variety. 
\end{remark}

\begin{conjecture}[$\LCT_n(I, J)$]\label{cnj:acc}
	If $I$ and $J$ satisfies the DCC, then $\LCT_n(I, J)$ satisfies the ACC.\\ 
\end{conjecture}
This conjecture is known in every dimension $n$ over algebraically closed fields of characteristic $0$ due to \cite{HMX14}. Over an algebraically closed field of characteristic $p>5$, it is known in dimension at most $3$ due to \cite{Bir16}.\\

\begin{theorem}[dlt-Model]\label{thm:dlt-model}
	Let $(X, \Delta)$ be a log canonical pair of dimension at most $3$ defined over a field $k$. If $\dim X=3$, then we further assume that $k$ is an algebraically closed field of $\chr p>5$; otherwise we assume that $k$ is an arbitrary field.\\
	Then there exits a birational morphism $f:(Y, \Delta_Y)\to (X, \Delta)$ extracting only exceptional divisors of discrepancy $a(E; X, \Delta)=-1$ such that $\Delta_Y$ is an effective $\mbR$-divisor, $Y$ is $\mbQ$-factorial, $(Y, \Delta_Y)$ has dlt singularities, and
	\[
		K_Y+\Delta_Y=f^*(K_X+\Delta).
	\]
\end{theorem}

\begin{proof}
	When $\dim X=3$, it is Theorem 1.6 in \cite{Bir16}. When $\dim X=2$, this follows from a standard application of the MMP (see \cite[Corollary 1.36]{Kol13}) by noticing that the MMP for excellent surfaces is known, thanks to Tanaka (see \cite[Theorem 1.1]{Tan18})
	
\end{proof}

\section{Log canonical thresholds in dimension one}\label{sec:acc-for-curves}
In this section we establish various ACC-type results for curves over arbitrary fields. These results are an important part of our argument in the main technical result, Theorem \ref{thm:main-technical-result} in Section \ref{sec:main-theorem}.\\ 

First we need the following very useful lemma.
\begin{lemma}\cite[Lemma 3.2]{BCZ18}\label{lem:imperfect-curve}
Let $C$ be a regular curve over an arbitrary field $k$. If $\deg_k K_C<0$ and $\ell=H^0(C, \mcO_C)$, then $C$ is a conic over $\ell$ and $\deg_\ell K_C=-2$. Furthermore, if $\chr(\ell)>2$, then $C_{\bar{\ell}}=C\times_\ell\bar{\ell}\cong\mbP^1_{\bar{\ell}}$.\\	
\end{lemma}

\begin{lemma}[ACC for log canonical thresholds for curves]\label{lem:acc-for-curves}
With the notations as in Definition \ref{def:acc-for-lct}, the ACC for log canonical thresholds holds in dimension $1$ over arbitrary fields, i.e., if $I\subset [0, 1]$ and $J\subset\mbR_{\>0}$ are two DCC sets then, $\LCT_1(I, J)$ satisfies the ACC. 
\end{lemma}

\begin{remark}
	Recall that here we do not fix the base field $k$, i.e., the base field $k$ may vary as the curves $C$ vary (see Remark \ref{rmk:lct}).  
\end{remark}

\begin{proof}
	On the contrary assume that it is not true, then there exist a sequence of pairs $(C_i, \Delta_i)\in\mfI_1(I)$ and effective $\mbR$-divisors $M_i$ with coefficients in $J$ such that $t_i=\lct(C_i, \Delta_i; M_i)$ is a strictly increasing sequence. Since $C_i$ is a regular curve and $t_i$ is a log canonical threshold, it follows that one of the points in the support of $M_i$ has coefficient $1$ in $\Delta_i+t_iM_i$. Let $\Delta_i=\sum a_{ij}x_{ij}$ and $M_i=\sum b_{ij}y_{ij}$, where $x_{ij}, y_{ij}$ are closed points of $C_i$. Without loss of generality we may assume that $a_{i1}+t_ib_{i1}=1$ for all $i\>1$. Since $\{a_{i1}\}$ and $\{b_{i1}\}$ are both contained in DCC sets, replacing them by a subsequence with common indices we may assume that they are both monotonically increasing sequence. Then $t_ib_{i1}=1-a_{i1}$ is monotonically decreasing. Since $\{t_i\}$ is strictly increasing it follows that $\{t_ib_{i1}\}$ is strictly increasing, hence a contradiction.\\  
\end{proof}

\begin{proposition}\label{pro:trivial-acc-for-curves}
	Fix a DCC set $I\subset [0, 1]$. Then there is a finite subset $I_0\subset I$ with the following properties:\\
	If $(C, \Delta)$ is a log pair such that
	\begin{enumerate}
		\item $C$ is a regular curve over some arbitrary field $k$,
		\item the coefficients of $\Delta$ belong to $I$, and
		\item $K_C+\Delta\num 0$,
	\end{enumerate}
	then the coefficients of $\Delta$ belong to $I_0$.
\end{proposition}

\begin{remark}
	Here we do not fix the base-field $k$, i.e., the base field $k$ may vary as $C$ varies.
\end{remark}

\begin{proof}
First we note that $(C, \Delta)$ is log canonical, since $C$ is a curve and $\Delta$ is a boundary divisor.	 Now it is enough to show that the coefficients of $\Delta$ belong to an ACC set. If not then assume that there is a strictly increasing sequence of coefficients, say
	\begin{equation}\label{eqn:not-acc-sequence}
		a_{11}<a_{21}<\cdots<a_{i1}<\cdots
	\end{equation}
	 where $\Delta_i=\sum_{j} a_{ij}x_{ij}$ and $(C_i, \Delta_i)$ is a pair as in the hypothesis.\\
	 
Suppose that $C_i$ is defined over the field $k_i$, and let $H^0(C_i, \mcO_{C_i})=\ell_i$. Then $\ell_i$ is a finite extension of $k_i$. From Lemma \ref{lem:imperfect-curve} it follows that $C_i$ is a conic over $\ell_i$ and $\deg_{\ell_i}K_{C_i}=-2$. Therefore $K_{C_i}+\Delta_i\num 0$ implies that $\deg_{\ell_i}(K_{C_i}+\Delta_i)=\deg_{k_i}(K_{C_i}+\Delta_i)/[\ell_i:k_i]=0$; in particular we have
\begin{equation}\label{eqn:main-acc-equation}
	-2+\sum_j n_{ij}a_{ij}=0,
\end{equation}
where $n_{ij}>0$ are positive integers (see Definition \ref{def:degree-on-curve}).\\
Then $n_{i1}a_{i1}+\sum_{j\>2}n_{ij}a_{ij}=2$. We claim that $\{n_{i1}:i\>1\}$ is a bounded set. Indeed, if not then there is an unbounded subsequence $\{n_{i_k1}\}_{k\>1}$. Since $\{a_{i1}\}_{i\>1}$ is contained in a DCC set, it has a non zero minimum, say $\min\{a_{i1}:i\>1\}=a>0$. Then for $k\gg 0$ we have $n_{i_k1}a_{i_k1}\>n_{i_k1}a>2$, which contradicts equation \eqref{eqn:main-acc-equation}.\\

\begin{claim}\label{clm:sum-dcc}
	$\left\{\sum_{j\>2}n_{ij}a_{ij}\right\}_{i\>1}$ is a DCC set.
\end{claim}
\begin{proof}[Proof of the Claim]
    Since $\{n_{i1}\}$ is a bounded sequence, from \eqref{eqn:main-acc-equation} it follows that $\sum_{j\>2}n_{ij}a_{ij}$ is also a bounded sequence. Note that if we can show that the number of prime components of the divisors $\Delta_i$ is bounded, then the claim will follow from Lemma \ref{lem:dcc-sets} part (5). To that end let $N_i$ be the number of prime components of $\Delta_i$ for all $i$. Let $a>0$ be the minimum of the set $\{a_{ij}: i\>1, j\>2\}$. Then $\sum_{j\>2}n_{ij}a_{ij}\>N_i a>0$, since $n_{ij}\>1$ for all $i, j$. Now if $\{N_i\}$ is unbounded, then $\{N_ia\}$ is unbounded, which contradicts the boundedness of $\sum_{j\>2}n_{ij}a_{ij}$. In particular, the claim follows from Lemma \ref{lem:dcc-sets} part (5) and (1).  
	
\end{proof}

\end{proof}

\begin{theorem}\label{thm:theta-equality-in-dim-1}
Fix a DCC set $I\subset\mbR_{\>0}$. Then there exists $0<\ve<1$ with the following properties: if $(C, \Theta)$ and $(C, \Theta')$ are two log pairs of dimension $1$  defined over some arbitrary field $k$ such that the coefficients of $\Theta$ belong to $I$, and 
\[
	(1-\ve)\Theta\<\Theta'\<\Theta,
\]
then $(C, \Theta)$ is log canonical if and only if $(C, \Theta')$ is log canonical. Moreover, if $(C, \Theta')$ is log canonical and $K_C+\Theta'\num 0$, then $\Theta'=\Theta$.
\end{theorem}

\begin{remark}
	Here we do not fix the base field $k$, i.e., the base field $k$ may vary as $C$ varies.
\end{remark}
\begin{proof}
	By Lemma \ref{lem:acc-for-curves} there exists an $\ve>0$ such that no elements of the set $\LCT_1(\{0\},I)$ is contained in the interval $[1-\ve, 1)$. Let $(C, \Theta)$ and $(C, \Theta')$ be two log pairs with coefficients of $\Theta$ in $I$ and $(1-\ve)\Theta\<\Theta'\<\Theta$. Clearly if $(C, \Theta)$ is log canonical, then so is $(C, \Theta')$. So assume that $(C, \Theta')$ is log canonical. Then $(C, (1-\ve)\Theta)$ is log canonical.
	Therefore $\lct((C, 0); \Theta)\>1-\ve$. Since the interval $[1-\ve, 1)$ does not contain any log canonical threshold, $\lct((C, 0); \Theta)\>1$; in particular $(C, \Theta)$ is log canonical.\\

For the second part by contradiction assume that the conclusion is false. Then there is a strictly decreasing sequence $\{\ve_i>0\}$ with $\lim \ve_i=0$ which satisfies the following properties: For each $i\>1$, there are log canonical pairs $(C_i, \Theta_i)$ and $(C_i, \Theta_i')$ such that the coefficients of $\Theta_i$ belong to $I$ and $K_{C_i}+\Theta_i'\num 0$, but
\begin{equation}\label{eqn:theta-equality}
	(1-\ve_i)\Theta_i\<\Theta_i'<\Theta_i.
\end{equation} 	
Write $\Theta_i=\sum d_{ij}D_{ij}$ and $\Theta_i'=\sum d'_{ij}D_{ij}$.\\

Next we make the following claim.
\begin{claim}\label{clm:dcc-coefficients}
	The coefficients of all $\Theta_i'$ are contained in a fixed DCC set. 
\end{claim}

Assuming the claim for the time being we will complete the proof first.\\
Let $J\subset [0, 1]$ be a DCC set containing the coefficients of $\Theta_i'$ for all $i$. Then by Proposition \ref{pro:trivial-acc-for-curves} the coefficients of $\Theta_i'$ are contained in a finite subset $J_0\subset J$. Relabeling the indices of $D_{ij}$ if necessary we may assume from \eqref{eqn:theta-equality} that 
\begin{equation}\label{eqn:theta-inequality}
	(1-\ve_i)d_{i1}\<d_{i1}'<d_{i1}\quad\mbox{ for all } i\>1.
\end{equation}

Now by passing to subsequences with common indices we may assume that $\{d_{i1}\}$ is monotonically increasing and $\{d_{i1}'\}$ is a constant sequence. Let $\lim d_{i1}=d$. Then from \eqref{eqn:theta-inequality} we get 
\[
	\lim(1-\ve_i)d_{i1}\<\lim d_{i1}'\<\lim d_{i1}\ \implies d\<d_{11}'\<d,\ \mbox{ i.e., } d=d_{11}'.	
\]
Then we have $d=\lim d_{i1}\>d_{11}>d_{11}'=d$, a contradiction.\\

\begin{proof}[Proof of the Claim \ref{clm:dcc-coefficients}]
To the contrary assume that the coefficients of $\Theta_i'$ are not contained in a fixed DCC set. Thus by relabeling the indices if necessary we may assume that $\{d_{i1}': i\>1\}$ is not contained in a DCC set. Then passing to subsequences with common indices we may assume that $\{d_{i1}'\}$ is strictly decreasing and $\{d_{i1}\}$ is monotonically increasing. Note that in this case we can only say that $(1-\ve_i)d_{i1}\<d_{i1}'\<d_{i1}$; strict inequality may not hold here. Let $\lim d_{i1}=d$ and $\lim d_{i1}'=d'$. Then $d_{i1}\<d$ and $d_{i1}'>d'$ for all $i\>1$. Thus $d\>d_{i1}\>d_{i1}'>d'=\lim d_{i1}'\>\lim (1-\ve_i)d_{i1}=d$, a contradiction.  
\end{proof}

\end{proof}

\section{Log canonical thresholds in dimension two and three}\label{sec:acc-for-surfaces-and-3-folds}
In this section we prove some results on the log canonical thresholds in dimension two and three and some other ACC-type results in positive characteristic over arbitrary fields.

\begin{theorem}\label{thm:acc-for-surfaces}
With notations as in Definition \ref{def:acc-for-lct}, the ACC for log canonical thresholds hold in dimension $2$ over arbitrary fields, i.e., if $I\subset [0, 1]$ and $J\subset\mbR_{\>0}$ satisfies the DCC, then  $\LCT_2(I, J)$ satisfies the ACC.\\
\end{theorem}

\begin{remark}
	Recall that here do not fix the base field $k$, i.e., the base field $k$ may vary as the surfaces vary (see Remark \ref{rmk:lct}).
\end{remark}

\begin{proof}
	This theorem is proved for surfaces defined over algebraically closed fields in \cite[Proposition 11.2]{Bir16}. In what follows we show that a similar argument works over arbitrary fields.\\
	
	By contradiction assume that there is a sequence of log canonical pairs $(X_i, \Delta_i)$ of dimension $2$ and effective $\mbR$-Cartier divisors $M_i\>0$ with coefficients of $\Delta_i$ in $I$ and the coefficients of $M_i$ in $J$ such that $t_i:=\lct(M_i; (X_i, \Delta_i))$ forms a strictly increasing sequence of real numbers.
	Now there are two cases based on the dimension of the log canonical centers.\\
	
	\textbf{Case I:} For infinitely many $i$, $(X_i, \Delta_i+t_iM_i)$ has a log canonical center of dimension $1$ contained in the $\Supp M_i$. In this case by passing to an infinite subsequence we may assume that for all $i\>1$, $(X_i, \Delta_i+t_iM_i)$ has a log canonical center of dimension $1$ contained in the $\Supp M_i$. In particular, one of the coefficients of $\Delta_i+t_iM_i$ is $1$ for every $i$. Let $\Delta=\sum a_{ij}D_{ij}$ and $M_i=\sum b_{ij}E_{ij}$. Without loss of generality assume that $a_{i1}+t_ib_{i1}=1$. Since $\{a_{i1}\}$ and $\{b_{i1}\}$ are contained in DCC sets, by passing to subsequences with common indices we may assume that $\{a_{i1}\}$ and $\{b_{i1}\}$ are both monotonically increasing. Then from $t_ib_{i1}=1-a_{i1}$ we see that the left hand side is strictly increasing sequence, while the right hand side is a monotonically decreasing sequence, hence a contradiction.\\  
	
	\textbf{Case II:} For infinitely many $i$, $(X_i, \Delta_i+t_iM_i)$ has a log canonical center $P_i$ of dimension $0$ contained in the support of $M_i$. By passing to an infinite subsequence we may assume that for all $i\>1$, $(X_i, \Delta_i+t_iM_i)$ has a log canonical center $P_i$ of dimension $0$ contained in the support of $M_i$. Let $f_i:Y_i\to X_i$ be a dlt-model of $(X_i, \Delta_i+t_iM_i)$ (see Theorem \ref{thm:dlt-model}). Then there is an exceptional divisor $E_i$ of discrepancy $-1$ which intersects the strict transform of $M_i$ and $f_i(E_i)=P_i$. Write
	\[
		K_{Y_i}+E_i+\Gamma_i=f^*_i(K_{X_i}+\Delta_i+t_iM_i).
	\]
Since the coefficients of $M_i$ are in a DCC set and $t_i$ is a strictly increasing sequence, it is easy to see that the coefficients of $t_iM_i$ are contained in a DCC set. Thus the coefficients of $\Delta_i+t_iM_i$ are in a DCC set by Lemma \ref{lem:dcc-sets}, part (4); in particular the coefficients of $\Gamma_i+E_i$ are in a DCC set. Since $(Y_i, \Gamma_i+E_i)$ is dlt, $E_i$ is a regular curve by \cite[Lemma 3.4]{BCZ18}. Then by adjunction we have $K_{E_i}+\Gamma_{E_i}=(K_{Y_i}+E_i+\Gamma_i)|_{E_i}$, and the coefficients of $\Gamma_{E_i}$ are in a DCC set by Proposition \ref{pro:adjunction}. Let $\ell_i=H^0(E_i, \mcO_{E_i})$. Then $\deg_{\ell_i}(K_{E_i}+\Gamma_{E_i})=((K_{Y_i}+E_i+\Gamma_i)\cdot_{k_i} E_i)/[\ell_i:k_i]=0$, where $X_i$ is defined over $k_i$. This implies that $\deg_{\ell_i}K_{E_i}<0$. Thus by Lemma \ref{lem:imperfect-curve} $E_i$ is a conic, and $\deg_{\ell_i}K_{E_i}=-2$. Then by Proposition \ref{pro:trivial-acc-for-curves} the coefficients of $\Gamma_{E_i}$ are contained in a finite set, say $I_0\subset [0, 1]$, since $K_{E_i}+\Gamma_{E_i}\num 0$ as $\deg_{\ell_i}(K_{E_i}+\Gamma_{E_i})=0$. By our construction and Proposition \ref{pro:adjunction}, for each $i\>1$ there is a number $1-\frac{1}{m_i}+t_i\frac{\sum_j n_{ij}b_{ij}}{m_i}$ contained in $I_0$ which is a coefficient of $\Gamma_{E_i}$, where $b_{ij}$'s are the coefficients of $\Gamma_i$ and thus contained in a DCC set. Since $I_0$ is a finite set, we may assume that $1-\frac{1}{m_i}+t_i\frac{\sum_j n_{ij}b_{ij}}{m_i}=b\in I_0$ for all $i$. Then $t_i\sum_j n_{ij}b_{ij}=m_i(b-1)+1$. If $b<1$, then $m_i=m_{i+1}$ for all $i\gg 0$, otherwise the right hand side generates a strictly decreasing sequence while the left hand side satisfies the DCC. If $b=1$, then $t_i\sum_j n_{ij}b_{ij}=1$ for all $i$. In either case we have $t_i\sum_j n_{ij}b_{ij}=b'$ for some constant $b'>0$ and for all $i\>1$. Now as in the proof of the Claim \ref{clm:sum-dcc} we can prove that the number of components of $\Gamma_{E_i}$ is bounded. Then $\sum_j n_{ij}b_{ij}$ satisfies DCC by Lemma \ref{lem:dcc-sets}, part(5). In particular, $\frac{1}{t_i}$ satisfies DCC, hence $t_i$ satisfies ACC, which is a contradiction.

\end{proof}

\begin{theorem}\cite[Proposition 11.7]{Bir16}\label{thm:trivial-acc-for-surfaces}
	Let $k$ be a fixed algebraically closed field of characteristic $p>0$, and $I\subset [0, 1]$ a DCC set. Then there exists a finite subset $I_0\subset I$ with the following properties:\\
	If $(X, \Delta)$ is a log pair such that
	\begin{enumerate}
		\item $X$ is a projective variety of dimension $2$ over $k$,
		\item $(X, \Delta)$ is log canonical,
		\item the coefficients of $\Delta$ belong to $I$, and 
		\item $K_X+\Delta\num 0$,
	\end{enumerate}
	then the coefficients of $\Delta$ belong to $I_0$.\\
\end{theorem}

\begin{theorem}\label{thm:theta-equality-in-dim-2}
	Fix a DCC set $I\subset \mbR_{\>0}$. Then there exists an $0<\ve<1$ with the following properties: If $(S, \Theta)$ and $(S, \Theta')$ are two $\mbQ$-Gorenstein log pairs of dimension $2$ over some arbitrary fields such that the coefficients of $\Theta$ belong to $I$, and 
	\[
		(1-\ve)\Theta\<\Theta'\<\Theta,
	\]
then $(S, \Theta)$ is log canonical if and only if $(S, \Theta')$ is log canonical. Moreover, if all the surfaces are defined over some fixed algebraically closed field $k=\overline{k}$ and $(S, \Theta')$ is log canonical and $K_S+\Theta'\num 0$, then $\Theta'=\Theta$.	 
\end{theorem}

\begin{remark}
In the first part we do not fix the base field, i.e., the base field may vary as the surfaces vary.
\end{remark}

\begin{proof}
	First note that if $(S, 0)$ is not log canonical, then the above statement is vacuously true. So assume that $(S, 0)$ is log canonical. Then by Theorem \ref{thm:acc-for-surfaces} there exists an $\ve>0$ such that no elements of the set $\LCT_2(\{0\}, I)$ is contained in the interval $[1-\ve, 1)$. Let $(S, \Theta)$ and $(S, \Theta')$ be two log pairs with coefficients of $\Theta$ in $I$ and $(1-\ve)\Theta\<\Theta'\<\Theta$. Clearly if $(S, \Theta)$ is log canonical, the so is $(S, \Theta')$. So assume that $(S, \Theta')$ is log canonical. Then $(S, (1-\ve)\Theta)$ is log canonical.
	Therefore $\lct((S, 0); \Theta)\>1-\ve$. Since the interval $[1-\ve, 1)$ does not contain any log canonical threshold, $\lct((S, 0); \Theta)\>1$; in particular $(S, \Theta)$ is log canonical.\\
	
	For the second part by contradiction assume that the conclusion is false. Then there is a strictly decreasing sequence $\{\ve_i>0\}$ with $\lim \ve_i=0$ which satisfies the following properties: For each $i\>1$, there are log canonical pairs $(S_i, \Theta_i)$ and $(S_i, \Theta_i')$ such that the coefficients of $\Theta_i$ belong to $I$ and $K_{S_i}+\Theta_i'\num 0$, but
\begin{equation}\label{eqn:theta-equality-in-dim-2}
	(1-\ve_i)\Theta_i\<\Theta_i'<\Theta_i.
\end{equation} 	
We remark that the rest of the proof works exactly as in the proof of Theorem \ref{thm:theta-equality-in-dim-1} by replacing the use of Proposition \ref{pro:trivial-acc-for-curves} by Theorem \ref{thm:trivial-acc-for-surfaces} and noticing the fact that the dimension of the ambient varieties were never used in the proof of Theorem \ref{thm:theta-equality-in-dim-1}.

\end{proof}

\begin{theorem}\label{thm:theta-comparison-in-dim-3}
	Let $k$ be a fixed algebraically closed field of characteristic $p>5$, and $I\subset\mbR_{\>0}$ a DCC set. Then there exists an $0<\ve<1$ with the following properties: If $(X, \Theta)$ and $(X, \Theta')$ are two $\mbQ$-factorial log pairs of dimension $3$ over $k$ such that the coefficients of $\Theta$ belong to $I$, and 
	\[
		(1-\ve)\Theta\<\Theta'\<\Theta,
	\]
then $(X, \Theta)$ is log canonical if and only if $(X, \Theta')$ is log canonical.	
\end{theorem}

\begin{proof}
	A similar proof as in the proof of the first part of the Theorem \ref{thm:theta-equality-in-dim-2} works here by noticing the fact that the ACC for log canonical thresholds is known for $3$-folds defined over an algebraically closed field of char $p>5$ due to Birkar, \cite[Theorem 1.10]{Bir16}.\\ 
\end{proof}

\section{Main Theorem}\label{sec:main-theorem}
First we prove the following main technical result.
\begin{theorem}\label{thm:main-technical-result}
Fix an algebraically closed field $k$ of characteristic $p>5$, and a DCC set $I\subset [0, 1]$. Let $\mfD$ be the set of all klt pairs $(X, \Delta)$ satisfying the following properties:
\begin{enumerate}
	\item $X$ is a projective $3$-fold defined over $k$,
	\item the coefficients of $\Delta$ belong to $I$,
	\item $K_X+\Delta\num 0$, and
	\item $\Delta$ is big.
\end{enumerate}
Then there exists a constant $0<\ve<1$ satisfying the following properties:\\
If $\Phi\>0$ is an effective $\mbR$-divisor such that $K_X+\Phi\num 0$ and $\Phi\>(1-\delta)\Delta$ for some $0<\delta<\ve$ and $(X, \Delta)\in\mfD$, then $(X, \Phi)$ has klt singularities.
\end{theorem}

\begin{remark}
	The statement above and its proof given below are both based on the proof of Lemma 6.1 in \cite{HMX14}. However, we note that our poof is \emph{significantly} more involved than that of \cite{HMX14}, having to do with the failure of Bertini's theorem for base-point free linear systems and other related results in positive characteristic. In a key argument in the proof of \cite[Lemma 6.1]{HMX14}, the authors reduce the problem to a Mori fiber space $g:Y\to Z$ and then restrict everything to the general fibers of $g$ which reduces the problem to a lower dimension. We face several challenges at this stage in positive characteristic. The first problem is that in positive characteristic the general fibers $F$ of a given fibration may have very bad singularities, they could in fact be \emph{non-reduced} schemes, even if the fibration is a Mori fiber space; this completely destroyes the hope of using general fibers. On top of that, even if we know that the general fibers $F$ of  $g:Y\to Z$ are reduced, irreducible and normal varieties, we still can not guarantee that $F$ has good MMP singularities via adjunction from $Y$; this has to do with the failure of \emph{generic smoothness} for fibrations in positive characteristic. In order to circumvent these issues, we work with the \emph{generic fiber} $Y_\eta$ of $g:Y\to Z$ instead of general fibers $F$. The generic fiber $Y_\eta$ is now a normal integral scheme over the function field $K(Z)$ of $Z$. However, this comes with a new set of challenges, since $Y_\eta$ is defined over $K(Z)$ which is an \emph{imperfect} field, lots of standard results are either not known for varieties over imperfect field or they are known to fail. Fortunately, in recent years lots of progress have been made towards understanding the birational geometry of \emph{surfaces} over imperfect fields, especially the minimal model program, mostly due to Tanaka, see \cite{Tan18}. We are able to use his results to our advantage to prove some ACC-type results for surfaces over arbitrary fields (see Section \ref{sec:acc-for-surfaces-and-3-folds}). In our proof we also have to deal with curves over imperfect fields and we need various ACC-type results on them as well, which are developed in Section \ref{sec:acc-for-curves}.\\
\end{remark}

\begin{proof}[Proof of Theorem \ref{thm:main-technical-result}]
	We fix an $\ve>0$ which is the minimum of the three values of $\ve$ obtained in Theorem \ref{thm:theta-equality-in-dim-1}, \ref{thm:theta-equality-in-dim-2} and \ref{thm:theta-comparison-in-dim-3}. Now by contradiction assume that there is a pair $(X, \Phi)$ such that $\Phi\>(1-\delta)\Delta$ for some $0<\delta<\ve$, $K_X+\Phi\num 0$ and $(X, \Phi)$ is not klt. Note that $\Phi>(1-\delta)\Delta$, since $(X, \Delta)$ is klt, and $\Phi$ is also big, since $\Delta$ is. Now we want to modify $\Phi$ so that $(X, \Phi)$ becomes log canonical but the other properties of $\Phi$ are preserved. If $(X, \Phi)$ is already log canonical, then there is nothing to do. So assume that $(X, \Phi)$ is not log canonical. Then by Lemma \ref{lem:convex-lct} there exists a $\lambda\in (0, 1)$ such that $(X, (1-\lambda)\Delta+\lambda\Phi)$ is log canonical but not klt. Observe that $((1-\lambda)\Delta+\lambda\Phi)>(1-\delta)\Delta$ and $K_X+(1-\lambda)\Delta+\lambda\Phi=(1-\lambda)(K_X+\Delta)+\lambda(K_X+\Phi)\num 0$. Thus by replacing $(1-\lambda)\Delta+\lambda\Phi$ by $\Phi$ we may assume that $(X, \Phi)$ is log canonical but not klt. Let $f:Y\to X$ be a dlt-model of $(X, \Phi)$, whose existence is guaranteed by \cite[Theorem 1.6]{Bir16}. We write
\begin{equation}\label{eqn:dlt-model}
	K_Y+\Psi=f^*(K_X+\Phi),
\end{equation}	
and
\begin{equation}\label{eqn:klt-model}
	K_Y+\Gamma+\sum a_iS_i=f^*(K_X+\Delta),
\end{equation}	
where $\lrd\Psi\rrd=\sum S_i$ and $a_i<1$ for all $i$, and $\Gamma\>0$ is an effective $\mbR$-divisor such that $\Supp\Gamma\subset\Supp(f^{-1}_*\Delta)$ and $\Gamma$ and $\lrd\Psi\rrd$ do not share any common component.\\ 

Since $K_Y+\Phi\num 0$, we have $K_Y+\Psi\num 0$. Thus by \cite[Theorem 1.7]{BW17} running a $(K_Y+\Psi-S_1)$-MMP we end up with a Mori fiber space $g:W\to Z$. Let $\phi:Y\bir W$ be induced birational map. Since $K_Y+\Psi\num 0$, every step of this MMP is $S_1$-positive, in particular, $\phi_*S_1$ is $g$-ample and hence $\phi_*S_1$ is not contracted by $g$. Observe that since $K_X+\Psi\num 0$ and $K_Y+\Gamma+\sum a_iS_i\num 0$ (as $K_X+\Delta\num 0$), these relations are preserved at every step of the $(K_Y+\Psi-S_1)$-MMP and eventually we have $K_W+\phi_*\Psi\num 0$ and $K_W+\phi_*\Gamma+\sum a_i\phi_*S_i\num 0$; this follows from \cite[Theorem 3.7(4)]{KM98} which in our case can be obtained from the cone theorem and base-point free theorem as in \cite[Theorem 1.1 and 1.2]{BW17}. It also follows from \cite[Lemma 3.38]{KM98} that $(W,\phi_*\Psi)$ is log canonical, however, note that it is not necessarily dlt.\\

Next we want to establish an inequality that $\Psi>(1-\ve)\Gamma+\sum S_i$; it will be used heavily in the rest of the proof. To this end we first show that $\Gamma\neq 0$. Indeed, if $\Gamma=0$, then $\Supp(f^{-1}_*\Delta)\subset\Supp\lrd\Psi\rrd$. Since $\Phi>(1-\ve)\Delta$, this implies that every component of $\Delta$ appears in $\Phi$ with coefficient $1$. In particular, $\Phi>\Delta$; but then $\Phi\num\Delta$ implies that $\Phi=\Delta$, which is a contradiction, since $(X, \Delta)$ is klt and $(X, \Phi)$ is not klt. Therefore $\Gamma>0$, and then from the inequality $\Phi>(1-\ve)\Delta$ and \eqref{eqn:dlt-model} and \eqref{eqn:klt-model} it follows that $\Psi>(1-\ve)\Gamma+\sum S_i$.\\

 In the following discussion we separate three cases based on the relative dimension of $g:W\to Z$. However, first we claim that we may assume that $\phi_*\Gamma$ is not contracted by $g:W\to Z$. Indeed, if $\phi_*\Gamma$ is contracted by $g$, then $\phi_*\Gamma\num_g 0$, since $\rho(W/Z)=1$. Again since $\rho(W/Z)=1$, any non-zero effective divisor on $W$ which is not contracted by $g$ is $g$-ample. In particular, from the discussion above it follows that $\sum \phi_*S_i$ is $g$-ample. Then $K_W+\sum\phi_*S_i\num_g K_W+\phi_*\Gamma+\sum \phi_*S_i\num \sum (1-a_i)\phi_*S_i $ is $g$-ample, since $a_i<1$ for all $i$. But then we have  $K_W+\phi_*\Psi\num 0$ and $\lrd\Psi\rrd=\sum S_i$, which is a contradiction. Therefore we may assume that $\phi_*\Gamma$ is not contracted by $g:W\to Z$ in the following discussion.\\

\textbf{Case I:} Relative dimension of $g:W\to Z$ is $1$. Let $F$ be the \emph{generic} fiber of $g$. Then Replacing $Y, \Gamma$ and $\Psi$ by $F$ and the restriction of $\phi_*\Gamma$ and $\phi_*\Psi$ to $F$, may assume that $Y$ is a regular curve over an imperfect field $\ell=K(Z)$, where $K(Z)$ is the function field of $Z$ (see Lemma \ref{lem:integral-generic-fiber} and \ref{lem:singularities-generic-fiber}). Further notice that $S_1$ and $\Gamma$ are non-zero effective divisors on $Y$, and $H^0(Y, \mcO_Y)=\ell$, since $g_*\mcO_W=\mcO_Z$, and $\deg_\ell K_Y<0$, since $K_Y+\Psi\num 0$ and $\Psi>0$. Thus by Lemma \ref{lem:imperfect-curve} $Y$ is a conic over $\ell$ and $\deg_\ell K_Y=-2$. Now $K_Y+\Gamma+\sum a_iS_i\num 0$ and $a_i<1$ for all $i$, so $\deg_\ell (K_Y+\Gamma+\sum S_i)>0$. By construction we also have $\Psi>(1-\ve)\Gamma+\sum S_i$. Thus for some $0<\eta<\ve$ we get that $\deg_\ell(K_Y+(1-\eta)\Gamma+\sum S_i)=\deg_\ell(K_Y+\Psi)=0$. But then we have
\begin{equation}\label{eqn:adjunction-in-relative-dim-1}
	(1-\ve)\left(\Gamma+\sum S_i\right)\<\left((1-\ve)\Gamma+\sum S_i\right)\<\left((1-\eta)\Gamma+\sum S_i\right)\<\left(\Gamma+\sum S_i\right).
\end{equation}
Since $(Y, \Psi)$ is log canonical by Lemma \ref{lem:singularities-generic-fiber} and $\Psi>(1-\ve)\Gamma+\sum S_i$, $(Y, (1-\ve)\Gamma+\sum S_i)$ is also log canonical. Then from \eqref{eqn:adjunction-in-relative-dim-1} and Theorem \ref{thm:theta-equality-in-dim-1} we get that $(Y, \Gamma+\sum S_i)$ is log canonical. In particular, then $(Y, (1-\eta)\Gamma+\sum S_i)$ is log canonical. Now observe that the coefficients of $\Gamma+\sum S_i$ are contained in $I\cup\{1\}$, which is a DCC set, hence by Theorem \ref{thm:theta-equality-in-dim-1}, $(1-\eta)\Gamma+\sum S_i=\Gamma+\sum S_i$, which is a contradiction, since $\deg_\ell((1-\eta)\Gamma+\sum S_i)=0$ and $\deg_\ell(K_Y+\Gamma+\sum S_i)>0$.\\

\textbf{Case II:} Relative dimension of $g:W\to Z$ is $2$. Let $F$ be the \emph{generic} fiber of $g$. Then replacing $Y, \Gamma$ and $\Psi$ by $F$ and the restriction of $\phi_*\Gamma$ and $\phi_*\Psi$ to $F$, we may assume that $Y$ is a normal $\mbQ$-factorial surface over an imperfect field $\ell=K(Z)$, where $K(Z)$ is the function field of $Z$ (see Lemma \ref{lem:integral-generic-fiber} and \ref{lem:singularities-generic-fiber}). We note that $S_1$ is an ample divisor, $\Gamma$ is a non-zero effective divisor (hence ample), and $H^0(Y, \mcO_Y)=\ell$ and the Picard number $\rho(Y)=1$ (see \cite[Lemma 6.6]{Tan15f}). We also note that $(Y, \Psi)$ is a log canonical pair by Lemma \ref{lem:singularities-generic-fiber}.\\

Since $Y$ is a surface let's rename $S_i$'s by $C_i$. Then $C_1$ is ample on $Y$. Furthermore, since $\rho(Y)=1$, any non-zero effective divisor is ample; in particular $\sum C_i$ is ample. Thus $K_Y+\Gamma+\sum C_i\num \sum (1-a_i)C_i$ is ample, since $a_i<1$ for all $i$. Since $\Psi>(1-\ve)\Gamma+\sum C_i$, $K_Y+\Psi\num 0$ and $\rho(Y)=1$, there exists an $0<\eta<\ve$ such that $K_Y+(1-\eta)\Gamma+\sum C_i\num 0$. We also note that $(Y, (1-\ve)\Gamma+\sum C_i)$ is log canonical, since $(1-\ve)\Gamma+\sum C_i<\Psi$.
 Let $C^n_1\to C_1$ be the normalization morphism. We have the following adjunction equations
 \begin{subequations}\label{eqn:dim-2-adjunction}
\begin{align}
	\left(K_Y+(1-\ve)\Gamma+\sum C_i\right)|_{C_1^n}& =K_{C^n_1}+\Theta_1,\\
	\left(K_Y+(1-\eta)\Gamma+\sum C_i\right)|_{C^n_1}& =K_{C^n_1}+\Theta_2,\quad\mbox{ and }\\
	\left(K_Y+\Gamma+\sum C_i\right)|_{C^n_1}& =K_{C^n_1}+\Theta.
\end{align}
\end{subequations}
Note that a priori it is not clear whether the coefficients of $\Theta$ are in the DCC set $D(I\cup\{1\})$, since we do not know wether $(Y, \Gamma+\sum C_i)$ is log canonical or not. In particular, some of the coefficients of $\Theta$ could potentially be larger than $1$. However, we claim that $(Y, \Gamma+\sum C_i)$ is indeed log canonical, and thus by Proposition \ref{pro:adjunction} the coefficients of $\Theta$ are in the DCC set $D(I\cup\{I\})$. The proof goes as follows: we have the following inequalities 
\[
	(1-\ve)\left(\Gamma+\sum C_i\right)\<(1-\ve)\Gamma+\sum C_i\<\Gamma+\sum C_i.
\]
Since $(Y, (1-\ve)\Gamma+\sum C_i)$ is log canonical, by Theorem \ref{thm:theta-equality-in-dim-2}, $(Y, \Gamma+\sum C_i)$ is log canonical.\\

Now from \eqref{eqn:dim-2-adjunction} we get that
\[
	(1-\ve)\Theta\<\Theta_1\<\Theta_2\<\Theta,
\]
where the first inequality follows from the fact that the coefficients of $\Theta$ belong to $D(I\cup\{1\})$ and the following inequality
\begin{equation}\label{eqn:dcc-comparison}
	t\left(\frac{m-1+f}{m}\right)\<\left(\frac{m-1+tf}{m}\right)\quad \mbox{ for any } t\<1.
\end{equation}
Now by adjunction (see Lemma \ref{lem:easy-adjunction}) $(C^n_1, \Theta_1)$ is log canonical. Since the coefficients of $\Theta$ are in a DCC set, by Theorem \ref{thm:theta-equality-in-dim-1} $(C^n_1, \Theta)$ is log canonical; in particular $(C^n_1, \Theta_2)$ is log canonical; hence again by Theorem \ref{thm:theta-equality-in-dim-1}, $\Theta_2=\Theta$, since $K_{C^n_1}+\Theta_2\num 0$. But this is a contradiction, since $K_{C^n_1}+\Theta$ is ample, as it is the pullback of an ample divisor by a finite morphism.\\

	\textbf{Case III:} Relative dimension of $g:W\to Z$ is $3$, i.e., $\dim Z=0$. Then replacing $Y, \Gamma$ and $\Psi$ by $W, \phi_*\Gamma$ and $\phi_*\Psi$ we may assume that $Y$ is a projective $3$-fold over the algebraically closed base field $k$, Picard number $\rho(X)=1$ and $S_1$ is an ample divisor on $Y$ and $\Gamma$ is a non-zero effective divisor (hence ample).\\

	Now since $\rho(Y)=1$, every non-zero effective divisor is ample. In particular, $K_Y+\Gamma+\sum S_i\num \sum (1-a_i)S_i$ is ample, since $a_i<1$ for all $i$. Also, since $K_Y+\Psi\num 0, \Psi>(1-\ve)\Gamma+\sum S_i$ and $\rho(X)=1$, it follows that there exists $0<\eta<\ve$ such that $K_Y+(1-\eta)\Gamma+\sum S_i\num 0$. We note that $(Y, (1-\ve)\Gamma+\sum S_i)$ is log canonical, since $(1-\ve)\Gamma+\sum S_i\<\Psi$. Let $S_1^n\to S_1$ be the normalization morphism. We have the following adjunction equations
	\begin{align*}
		\left(K_Y+(1-\ve)\Gamma+\sum S_i\right)|_{S_1^n}& =K_{S_1^n}+\Theta_1,\\
		\left(K_Y+(1-\eta)\Gamma+\sum S_i\right)|_{S_1^n}& =K_{S_1^n}+\Theta_2,\quad\mbox{ and }\\
		\left(K_Y+\Gamma+\sum S_i\right)|_{S_1^n}& =K_{S_1^n}+\Theta.
	\end{align*}
	As in the proof of Case II using Theorem \ref{thm:theta-comparison-in-dim-3} we see that $(Y, \Gamma+\sum S_i)$ is log canonical; in particular the coefficients of $\Theta$ are in the DCC set $D(I\cup\{1\})$ by Proposition \ref{pro:adjunction}. The rest of the arguments are verbatim to the Case II via Lemma \ref{lem:q-gorenstein} and Theorem \ref{thm:theta-equality-in-dim-2}.\\

\end{proof}

\begin{proof}[Proof of Theorem \ref{thm:main-theorem}]
	Let $\ve>0$ be a constant given by Theorem \ref{thm:main-technical-result}. We claim that
	\[
		\vol_X(-K_X)=\vol_X(\Delta)\<\frac{27}{\ve^3}\quad \mbox{ for all } (X, \Delta)\in\mfD.
	\]
By contradiction assume that there is a $(X, \Delta)\in\mfD$ such that $\vol_X(X, \ve\Delta)>3^3$. Since the volume is a continuous function (see \cite[Theorem 2.2.44]{Laz04a}), we have
\[
	\vol_X(\eta\Delta)>3^3\quad \mbox{ for some } 0<\eta<\ve.
\]	 
Note that $(X, (1-\eta)\Delta)$ klt, since $(X, \Delta)$ is klt. Let $x\in X$ be a smooth closed point. Then by Lemma \ref{lem:klt-perturbation} there exists an effective $\mbR$-divisor $\Pi$ passing through $x$ such that $\Pi\sim_\mbR\eta\Delta$ and $(X, (1-\eta)\Delta+\Pi)$ is not klt. This is a contradiction to Theorem \ref{thm:main-technical-result}, since $0<\eta<\ve$.\\

\end{proof}



\end{document}